\documentclass[letterpaper]{amsart}

\usepackage[utf8]{inputenc}
\usepackage[T1]{fontenc}
\usepackage{lmodern}
\usepackage{amssymb}
\usepackage{amscd}
\usepackage[pdfusetitle]{hyperref}

\def\arxiv#1{\href{http://arxiv.org/abs/#1}{\texttt{arXiv:#1}}}
\def\doi#1{\href{http://dx.doi.org/#1}{doi:#1}}

\theoremstyle{plain}
\newtheorem{theorem}{Theorem}[section]
\newtheorem{proposition}[theorem]{Proposition}
\newtheorem{lemma}[theorem]{Lemma}

\theoremstyle{definition}

\newtheorem{remark}[theorem]{Remark}
\newtheorem{remarks}[theorem]{Remarks}

\theoremstyle{remark}

\numberwithin{equation}{section}

\def\C{\mathbb C}

\let\phi\varphi

\def\N{\mathbb N}
\def\Z{\mathbb Z}
\def\Q{\mathbb Q}

\def\CP{\mathbb{CP}}
\def\PP{\mathbb P}

\def\Cstar{\C^{\times}}
\def\pcont#1#2{{}_{#2}#1}

\DeclareMathOperator{\Hom}{Hom}

\def\cf{\emph{cf.}}

\def\cP{\mathcal P}
\def\cQ{\mathcal Q}

\def\larrow#1{\to}

\begin{document}

\title{The classification of weighted projective spaces}

\author[A.~Bahri]{Anthony Bahri}
\address{Department of Mathematics, Rider
  University, Lawrenceville, NJ~08648, U.S.A.}
\email{bahri@rider.edu}
\author[M.~Franz]{Matthias Franz}
\address{Department of Mathematics, University of Western Ontario, London, Ont.\ N6A\;5B7, Canada}
\email{mfranz@uwo.ca}
\author[D.~Notbohm]{Dietrich Notbohm}
\address{Department of Mathematics, Vrije Universiteit Amsterdam, De Boelelaan 1081a, 1081~HV Amsterdam, The Netherlands}
\email{notbohm@few.vu.nl}
\author[N.~Ray]{Nigel Ray}
\address{School of Mathematics, University of Manchester, Oxford Road, Manchester M13~9PL, United Kingdom}
\email{nigel.ray@manchester.ac.uk}

\pdfstringdefDisableCommands{\def\and{, }}
\hypersetup{pdfauthor=\authors}


\begin{abstract}
  We obtain two classifications of weighted projective spa\-ces; up
  to homeomorphism and up to homotopy equivalence. We show that the 
  former coincides with Al~Amrani's classification up to isomorphism of 
  algebraic varieties, and deduce the latter by proving that the Mislin genus 
  of any weighted projective space is rigid.
\end{abstract}

\maketitle

\section{Introduction}

Weighted projective spaces are the simplest projective toric varieties
that exhibit orbifold singularities.  They have been extensively
investigated by algebraic geometers, but have attracted only fleeting
attention from algebraic topologists
since Kawasaki's pioneering work~\cite{Kawasaki:1973}, in which he
computed their integral cohomology rings. Subsequently, their
$K$-theory was determined by Al~Amrani~\cite{AlAmrani:1994}, and the
study of their $KO$-theory was initiated by
Nishimura--Yosimura~\cite{NishimuraYosimura:1997}.

In toric geometry, weighted projective spaces are classified by their
fans. Here, we give two classifications that are fundamental to
algebraic topology: up to homeo\-morphism, and up to homotopy
equivalence. We obtain the second as a consequence of the
fact that the Mislin genus of a weighted projective space is rigid.
Our results are stated below, following summaries of the definitions and
notation.

A \emph{weight vector}~$\chi=(\chi_{0},\dots,\chi_{n})$ is a finite
sequence of positive integers.  It gives rise to a weighted action
of~$S^{1}$ on~$S^{2n+1}\subset\C^{n+1}$,
\begin{equation}
  \label{eq:weighted-action}
  g\cdot z = \bigl(g^{\chi_{0}}z_{0},\dots,g^{\chi_{n}}z_{n}\bigr)
  \qquad \hbox{for~$g\in S^{1}, z\in S^{2n+1}$.}
\end{equation}
The quotient~$S^{2n+1}/S^{1}\langle \chi\rangle$ is the weighted
projective space~$\PP(\chi)$.  Alternatively, $\PP(\chi)$ may be
defined as the quotient of~$\C^{n+1}\setminus\{0\}$ by the same
weighted action of~$\Cstar$; this exhibits $\PP(\chi)$ as a complex
projective variety.

Scaling the weight vector~$\chi$ leads to isomorphic weighted projective
spaces~$\PP(\chi)$ and~$\PP(m\chi)$, for any integer~$m\geq 1$.
Moreover,
if all weights except, say, $\chi_{0}$ are divisible by some prime~$p$,
then the map
\begin{equation}
  \label{eq:normiso}
  \PP(\chi) \to \PP(\chi_{0},\chi_{1}/p,\dots,\chi_{n}/p),
  \quad
  \bigl[z_{0}:\dots:z_{n}\bigr] \mapsto \bigl[z_{0}^{p}:z_{1}:\dots:z_{n}\bigr]
\end{equation}
is an isomorphism as well, \cf~\cite[\S 5.7]{Fletcher:2000}.  This
leads to the notion of \emph{normalized weights}: a weight
vector~$\chi$ is normalized if for any prime~$p$ at least two weights
in~$\chi$ are not divisible by~$p$.  Any weight vector can be
transformed to a unique normalized vector by repeated application of
scaling and \eqref{eq:normiso}. Consequently, two weighted projective 
spaces are isomorphic as algebraic varieties and homeomorphic 
as topological spaces if they have the same normalized weights, 
up to order. We prove that the converse is also true.
In particular, we recover Al~Amrani's classification up to 
isomorphism of algebraic varieties \cite[\S 8.1]{AlAmrani:1989}.

\begin{theorem}
  \label{thm:classification-homeomorphism}
  The following are equivalent for any weight vectors $\chi$~and~$\chi'$:
  \begin{enumerate}
  \item \label{thm:a1}
    The normalizations
    of $\chi$~and~$\chi'$ are the same, up to order.
  \item \label{thm:a2}
    $\PP(\chi)$~and~$\PP(\chi')$ are isomorphic as algebraic varieties.
  \item \label{thm:a3}
    $\PP(\chi)$~and~$\PP(\chi')$ are homeomorphic.
  \end{enumerate}
\end{theorem}

For any prime $p$, the \emph{$p$-content}~$\pcont{\chi}{p}$ of~$\chi$
is the vector made up of the highest powers of~$p$ dividing the
individual weights.  For example, $\pcont{(1,2,3,4)}{2}=(1,2,1,4)$.
Let $\chi$~and~$\chi'$ be two normalized weight vectors.  It follows
from Kawasaki's result that the cohomology rings
$H^{*}(\PP(\chi);\Z)$~and~$H^{*}(\PP(\chi');\Z)$ are isomorphic if and
only if, for all primes~$p$, the $p$-contents
$\pcont{\chi}{p}$~and~$\pcont{\chi'}{p}$ are the same up to order.
The same phenomenon can be observed in $K$-theory and $KO$-theory. In
fact, no cohomology theory can tell such spaces apart:

\begin{theorem}
  \label{thm:classification-homotopy}
  Two weighted projective spaces are homotopy equivalent
  if and only if for all primes~$p$, the $p$-contents of
  their normalized weights are the same, up to order.
\end{theorem}

The torus~$T = (S^{1})^{n+1}/S^{1}\langle\chi\rangle \cong (S^{1})^{n}$
and its complexification~$T_{\C}$ act on~$\PP(\chi)$ in a canonical
way, and the resulting equivariant homotopy type is a
finer invariant. As shown in~\cite[Thm.~5.1]{BahriFranzRay:2009},
the equivariant cohomology ring~$H_{T}^{*}(\PP(\chi);\Z)$
determines the normalized weights up to order.

Let $\pcont{\chi^{*}}{p}$ be the vector obtained from $\pcont{\chi}{p}$ 
by ordering its coordinates as a 
non-decreasing sequence, and let $\chi^{*}$ denote the product of 
the~$\pcont{\chi^{*}}{p}$, taken coordinatewise. For
example, $(1,2,3,4)^*=(1,1,2,12)$.
By Theorem~\ref{thm:classification-homotopy},
$\PP(\chi)$ is homotopy equivalent to~$\PP(\chi^{*})$.
The weights in~$\chi^{*}$ form a
\emph{divisor chain}, in the sense that each weight divides the next.
As a consequence, the space~$\PP(\chi^{*})$ is particularly easy to 
work with because the differences
\begin{equation}
  \label{eq:cell-decomposition}
  * = \PP(\chi^{*}_{n}), \;\;
  \PP(\chi^{*}_{n-1},\chi^{*}_{n}) \setminus \PP(\chi^{*}_{n}),\;\;
  \dots,\;\;
  \PP(\chi^{*}_{0},\dots,\chi^{*}_{n}) \setminus
  \PP(\chi^{*}_{1},\dots,\chi^{*}_{n})
\end{equation}
form a cell decomposition of~$\PP(\chi^{*})$
(see Remark~\ref{rem:cell-decomposition} below), and
\begin{equation}
  * = \PP(\chi^{*}_{0}) \subset \PP(\chi^{*}_{0},\chi^{*}_{1})
  \subset \cdots \subset
  \PP(\chi^{*}_{0},\dots,\chi^{*}_{n-1}) \subset
  \PP(\chi^{*}_{0},\dots,\chi^{*}_{n})
\end{equation}
displays $\PP(\chi^{*})$ as an iterated Thom space 
\cite[Cor.~3.8]{BahriFranzRay:2011}.

The \emph{Mislin genus} of a weighted projective space~$\PP(\chi)$ is
the set of all homotopy classes of simply connected CW~complexes~$Y$
of finite type such that for all primes~$p$ the $p$-localizations
of~$Y$~and~$\PP(\chi)$ are homotopy equivalent.  The Mislin genus of a
space is \emph{rigid} if it contains only the class of the space itself.

\begin{theorem}\label{thm:genus-wps}
  The Mislin genus of any weighted projective space is rigid.
\end{theorem}

For~$\CP^{n}$, this has been established by
McGibbon~\cite[Thm.~4.2\,(ii)]{McGibbon:1982}.

\medbreak

In Section~\ref{sec:kawasaki} we review Kawasaki's results on which
our work is based. Theorem~\ref{thm:classification-homeomorphism} is
proved in Section~\ref{sec:homeo}, and Theorems 
\ref{thm:classification-homotopy} and \ref{thm:genus-wps} in
Section~\ref{sec:homotopy}; necessary conditions for the 
rigidity of the Mislin genus are established in Section \ref{sec:mislin}.

\section{Kawasaki's results}
\label{sec:kawasaki}

From now on, $\chi=(\chi_{0},\dots,\chi_{n})$ always denotes a
normalized a weight vector, and cohomology is taken with integer
coefficients unless otherwise stated. In order to make Kawasaki's
description of $H^{*}(\PP(\chi))$ explicit, it is convenient to recall 
his notation $(r_{0}(\chi;p),\dots,r_{n}(\chi;p))$ for the 
non-decreasing weight vector $\pcont{\chi^{*}}{p}$; given any 
$0\le i\le n$, we then set
\begin{equation}
  l_{i} = l_{i}(\chi) =
  \prod_{\text{$p$ prime}}\! r_{n-i+1}(\chi;p)\cdots r_{n}(\chi;p).
\end{equation}
We also consider the map
\begin{equation}
  \label{eq:definition-phi}
  \phi=\phi_{\chi}\colon \CP^{n}\to\PP(\chi),
  \quad
  [z_{0}:\dots:z_{n}] \mapsto [z_{0}^{\chi_{0}}:\dots:z_{n}^{\chi_{n}}].
\end{equation}

\begin{theorem}[{\cite[Thm.~1]{Kawasaki:1973}}]
  \label{thm:kawasaki-wps}
  Additively, $H^{*}(\PP(\chi))\cong H^{*}(\CP^{n})$. Furthermore, 
  there exist generators $\xi_{i}\in H^{2i}(\PP(\chi))$ and
  $\eta\in H^{2}(\CP^{n})$ such that $\phi^{*}(\xi_{i})=l_{i}\eta^{i}$
  for $0\le i\le n$; the multiplicative structure is specified by
  \begin{equation*}
    \xi_{i}\xi_{j} = \frac{l_{i}l_{j}}{l_{i+j}}\,\xi_{i+j}
  \end{equation*}
in $H^{2(i+j)}(\PP(\chi))$, for $0\le i+j\le n$.
\end{theorem}

\begin{remarks}\label{rem:simfree}
Kawasaki's proof of Theorem \ref{thm:kawasaki-wps} shows that the
integral homology groups $H_*(\PP(\chi))$ are finitely generated and 
torsion-free, and therefore isomorphic to $\Hom(H^*(\PP(\chi)),\Z)$ 
by the Universal Coefficient Theorem.

Moreover, \cite[Sec.~3.2]{Fulton:1993} and
\cite[Cor.~7.2]{Illman:1983} confirm that $P(\chi)$ is a simply
connected finite CW complex, for every choice of $\chi$.
\end{remarks}

Kawasaki also determined the cohomology of the generalized lens
space~$L(k;\chi) = S^{2n+1}/\Z_{k}\langle\chi\rangle$, where in this
case $\chi$ describes the weights of the $k$-th roots of unity.  The
answer depends on the augmented weight
vector~$(\chi,k)=(\chi_{0},\dots,\chi_{n},k)$.

\begin{theorem}[{\cite[Thm.~2]{Kawasaki:1973}}]
  \label{thm:kawasaki-lens}
  The non-zero cohomology groups of~$L=L(k;\chi)$ are
  $H^{0}(L)\cong H^{2n+1}(L)\cong\Z$ and
  $H^{2i}(L) \cong \Z_{q}$ for~$1\le i\le n$, where
  $q = l_{i}(\chi,k)/l_{i}(\chi)$.
\end{theorem}

\section{Classification up to homeomorphism}
\label{sec:homeo}

\def\kk{q}
\def\qq{q'}

Consider a point~$z\in\PP(\chi)$.
Let $I$~and~$J$ be the subsets of~$\{0,\dots,n\}$ corresponding
to the zero and non-zero homogeneous coordinates of~$z$, respectively,
and let $\kk=\gcd\{\chi_{i}:i\in J\}$.
Also, let $U_{I}=\{ [z_{0}:\dots:z_{n}]:\text{$z_{i}\ne 0$ for $i\notin I$}\}$,
and write $\chi_{I}\in\Z^{I}$ for the weights indexed by~$I$.

\begin{lemma}[{\cf~\cite[\S 5.15]{Fletcher:2000}}]
  \label{wps-to-lens-general}
  There is an isomorphism of algebraic varieties
  \begin{equation*}
    U_{I} \cong (\Cstar)^{|J|-1}\times \C^{I} \!/\, 
  \Z_{\kk}\langle \chi_{I}\rangle,
  \end{equation*}
  sending
  $z$ to a point of the form~$(\tilde z,0)$.
\end{lemma}

Observe that $\C^{I} / \Z_{\kk}\langle \chi_{I}\rangle$ is the
unbounded cone over $L(\kk;\chi_{I})$.

\begin{proof}
  The weight vector~$\chi_{J}$ determines a morphism~$\Cstar\to(\Cstar)^{J}$
  with kernel~$\Z_{\kk}$. Let $T'$ be its image and $T''\cong(\Cstar)^{|J|-1}$
  a torus complement. Then
\begin{equation*}
    U_{I} =
    \bigl( (\Cstar)^{J}\times \C^{I}\bigr)\,\big/\,\Cstar\langle \chi\rangle
    = \bigl(T''\times T'\times\C^{I} \bigr)\,\big/\,\Cstar\langle \chi\rangle
    = T''\times \C^{I}\!/\,\Z_{\kk}\langle \chi_{I}\rangle.
    \qedhere
  \end{equation*}
\end{proof}

\begin{remark}
  \label{rem:cell-decomposition}
  If $\chi_{0}=1$ and $z=[1:0:\dots:0]$, then $U_{I}\cong\C^{n}$.  If
  the weights form a divisor chain, we have $\PP(\chi)\setminus
  U_{I}=\PP(\chi_{1},\dots,\chi_{n})=
  \PP(1,\chi_{2}/\chi_{1},\dots,\chi_{n}/\chi_{1})$;
  hence we obtain an inductive decomposition of~$\PP(\chi)$ into
  $n+1$~cells~$*$,~$\C$,~$\C^{2}$,~\dots,~$\C^{n}$.
\end{remark}

\begin{lemma}
  \label{local-homology-wps}
  There is an isomorphism 
   $H^{2n-1}\bigl(\PP(\chi),\PP(\chi)\setminus\{z\}\bigr) 
   \cong \Z_{\kk}$.
\end{lemma}
\begin{proof}
  Set $X = (\Cstar)^{|J|-1}$, $Y = \C^{I} / \Z_{\kk}\langle \chi_{I}\rangle$
  and $m=|I|-1$.
  Note that $X$ is a manifold of dimension~$2(n-m-1)$, 
  so that $H^{*}(X,X\setminus\{\tilde z\})$ is isomorphic to~$\Z$
  in dimension~$2(n-m-1)$ and zero otherwise. 
  Excision, Lemma~\ref{wps-to-lens-general} and the Künneth formula for
  relative cohomology therefore imply
  \begin{align*}
    H^{*}\bigl(\PP(\chi),\PP(\chi)\setminus\{z\}\bigr)
    &\cong H^{*}\bigl(U_{I},U_{I}\setminus\{z\}\bigr) \\
    &\cong H^{*}(X\times Y, (X\setminus\{\tilde z\})
    \times Y \cup X\times(Y\setminus\{0\})) \\
    &\cong H^{*}(X,X\setminus\{\tilde z\})\otimes H^{*}(Y,Y\setminus\{0\}),\\
\intertext{because $H^{*}(X,X\setminus\{\tilde z\})$ is free. 
    In particular,}
    H^{2n-1}\bigl(\PP(\chi),\PP(\chi)\setminus\{z\}\bigr)
    &\cong H^{2m+1}\bigl(Y,Y\setminus\{0\}\bigr)
    \cong \tilde H^{2m}(L(\kk;\chi_{I})).
  \end{align*}

  If $m=0$, then $\kk=1$ because $\chi$ is normalized, and the claim holds.
  Otherwise, Theorem~\ref{thm:kawasaki-lens} gives
  $H^{2m}(L(\kk;\chi_{I})) \cong \Z_{\qq}$, where
  the $p$-content of~$\qq$ is given by
  \begin{equation}
    \label{eq:product-r}
    \hbox{$p$-content of} \;\; \frac{l_{m}(\chi_{I},\kk)}{l_{m}(\chi_{I})}
    = \prod_{i=1}^{m} \frac{r_{m+2-i}(\chi_{I},\kk;p)}{r_{m+1-i}(\chi_{I};p)}.
  \end{equation}
  We have to show $\qq=\kk$, which means that $\qq$~and~$\kk$ have the
  same $p$-content for all~$p$.  This is clearly true if $\kk$ is not
  divisible by~$p$.  Otherwise, $\chi_{I}$ inherits from the
  normalized weight vector~$\chi$ two weights not divisible by~$p$.
  (Recall that $\kk$ is the gcd of the weights appearing in~$\chi$,
  but not in~$\chi_{I}$.)  Hence, $r_{1}(\chi;p)=1$, and the numerator
  of~\eqref{eq:product-r} differs from the denominator by the
  $p$-content of~$\kk$. 
  This finishes the proof.
\end{proof}

\begin{proof}[Proof of Theorem~\ref{thm:classification-homeomorphism}]
  By the remarks preceding the theorem, we only have to prove the
  implication~$\eqref{thm:a3}\Rightarrow\eqref{thm:a1}$.  In order to
  do so, we show how to read off the normalized weights from
  topological invariants of a weighted projective space~$\PP(\chi)$.
   For~$z\in\PP(\chi)$, let $\qq(z)$ be the order of the finite
  group~\smash{$H^{2n-1}\bigl(\PP(\chi),\PP(\chi)\setminus\{z\}\bigr)$}.
  Lemma~\ref{local-homology-wps} implies that for all~$d\ge 1$ the
  space
  \begin{equation*}
    X(d) = \bigl\{ z\in\PP(\chi) : d \mid \qq(z) \bigr\}
  \end{equation*}
  is again a weighted projective space or empty. In fact,
  \begin{equation*}
    X(d) = \bigl\{ [z_{0}:\dots:z_{n}] \in\PP(\chi) :
      \text{$z_{i} = 0$ if $d \nmid \chi_{i}$} \bigr\}
  \end{equation*}
  because $d$ divides $\qq(z)=\kk$ if and only if it divides
  $\chi_{i}$ for all~$i$ such that $z_{i}\ne0$. For each~$d$, the
  dimension of~$X(d)$ (which equals the degree of the highest
  non-vanishing cohomology group) therefore tells us the number of
  weights divisible by~$d$.  This determines the normalized weights
  completely up to order.
\end{proof}

\section{The Mislin genus}\label{sec:mislin}

This section relies heavily on the theory of localization and 
homotopy pullbacks. We refer readers to 
\cite{HiltonMislinRoitberg:1975}, especially Chapter~II, and 
to \cite[Chap.~7]{Strom:2011}, for background information.

Throughout the section, $X$, $Y$, and $Z$ denote simply connected 
CW~complexes. A map~$f\colon X\to Y$ is therefore a homotopy equivalence 
(written $X\simeq Y$) if and only if it induces an isomorphism~$H_*(f)$ 
of integral homology; in this case, $f^{-1}$ denotes a homotopy inverse 
for~$f$. 

Given any set~$\cP$ of primes, the algebraic localization of~$\Z$
is denoted by~$\Z_{\cP}$, and the homotopy theoretic localization of
$X$ by $X_{\cP}$; the latter is also a CW complex. Every~$X$ admits a
localization map~$l_{\cP}\colon X\to X_{\cP}$, which induces an
isomorphism $H_*(l_{\cP};\Z_{\cP})$, and every $f$ admits a 
localization $f_{\cP}\colon X_{\cP}\to Y_{\cP}$, for which the square
\begin{equation*}
 \begin{CD}
   X@>f>>Y\\
     @Vl_{\cP}VV@VVl_{\cP}V\\
     X_{\cP}@>>f_{\cP}>Y_{\cP}
 \end{CD}
\end{equation*}
is homotopy commutative. Any map~$g\colon X_{\cP}\to Y_{\cP}$ of 
localized spaces is a homotopy equivalence if and only if it 
induces an isomorphism $H_*(g;\Z_{\cP})$.

If $\cP$ is empty, then the localization map $l_\emptyset$ is 
\emph{rationalization}, and is denoted by 
$l_0\colon X\to X_0$; likewise, $f_\emptyset$ is denoted by
$f_0\colon X_0\to Y_0$. If $\cP$ consists of a single prime $p$, then 
the localization map $l_{\{p\}}$ is abbreviated to $l_p\colon X\to X_p$, 
and $f_{\{p\}}$ is abbreviated to~$f_p\colon X_p\to Y_p$. If 
$\cP$ contains all primes, then $l_{\cP}$ is a homotopy equivalence.

The \emph{homotopy pullback} of a diagram~$X\to Z\leftarrow Y$ 
may be constructed by replacing either map with a fibration,
and pulling it back along the other in the standard fashion. 
The resulting square is unique up to equivalence of diagrams in the 
homotopy category \cite[\S 7.3]{Strom:2011}; in particular, the 
homotopy pullback is well-defined up to homotopy equivalence.

For any set of primes $\cP$, the rationalization of $X_{\cP}$ is 
homotopy equivalent to $X_0$, so the rationalization map may be 
expressed as $l_0\colon X_{\cP}\to X_0$. If $\cP$~and~$\cQ$ are 
disjoint, then the homotopy pullback of
\begin{equation}\label{eq:locpq}
    \begin{CD}
      X_{\cQ}@>{l_0}>>X_0@<l_0<<X_{\cP}
    \end{CD}\,
  \end{equation}
 is $X_{\cP\cup\cQ}$ (see \cite[Prop.~2.9.3]{Neisendorfer:2010} 
 or~\cite[proof of Thm.~7.13]{HiltonMislinRoitberg:1975}). 

\begin{lemma}\label{thm:homotopy-pullback}
  Given two disjoint sets $\cP$ and~$\cQ$ of primes,
  let $f\colon Y_{\cP}\larrow{\simeq} Z_{\cP}$ and 
  $g\colon Y_{\cQ}\larrow{\simeq} Z_{\cQ}$ be
  homotopy equivalences, and define $h=f_0g_0^{-1}$; 
  then $Y_{\cP\cup\cQ}$ is the homotopy pullback of the diagram
  \begin{equation}\label{eq:homotopy-pullback}
    \begin{CD}
      Z_{\cQ}@>{hl_0}>>Z_0@<l_0<<Z_{\cP}
    \end{CD}\,.
  \end{equation}
  If also there exist homotopy equivalences 
  $d\colon Z_{\cP}\to Z_{\cP}$ and~$e\colon Z_{\cQ}\to Z_{\cQ}$
  such that $h\simeq d_0e_0^{-1}$, then $Y_{\cP\cup\cQ}$ and 
  $Z_{\cP\cup\cQ}$ are homotopy equivalent.
\end{lemma}

\begin{proof}
  The vertical maps in the homotopy commutative ladder
  \begin{equation}\label{eq:ladder}
    \begin{CD}
      Y_{\cQ}@>{l_0}>>Y_0@<{l_0}<<Y_{\cP}\\
      @VgVV@VV{f_0}V@VVfV\\
      Z_{\cQ}@>>{hl_0}>Z_0@<<{l_0}<Z_{\cP}
    \end{CD}
  \end{equation}
  are homotopy equivalences, and the homotopy pullback of the upper 
  row is $Y_{\cP\cup\cQ}$, by analogy with \eqref{eq:locpq}.
  So the ladder induces a homotopy equivalence of homotopy pullbacks
  following \cite[\S7.3]{Strom:2011}, and the first claim follows.

  Substituting $Y=Z$, $f=d$ and $g=e$ into \eqref{eq:ladder} 
  creates an upper row with homotopy pullback $Z_{\cP\cup\cQ}$ . The 
  second claim is then immediate.
\end{proof}

The following proposition gives criteria for ensuring that the genus of 
a finite CW~complex is rigid.
\begin{proposition} \label{thm:genus-rigid}
  Let $Z$ be a simply connected finite CW~complex satisfying
  \begin{itemize}
  \item[(i)] for any space~$Y$ in the Mislin genus of~$Z$, there exists
   a rational homotopy equivalence $k\colon Y\to Z$, and
  \item[(ii)] for any disjoint sets~$\cP$ and $\cQ$ of primes, and any
  rational homotopy equivalence~$h\colon Z_0\to Z_0$, there exist 
  homotopy equivalences $d\colon Z_{\cP}\to Z_{\cP}$ 
  and $e\colon Z_{\cQ}\to Z_{\cQ}$ such that $h\simeq d_0e^{-1}_0$:
  \end{itemize}
then the genus of~$Z$ is rigid.
\end{proposition}
\begin{proof}
  Let $Y$ belong to the Mislin genus of~$Z$. It follows from
  \cite[p.\,105]{HiltonMislinRoitberg:1975} that there is an 
  isomorphism $H_*(Y;\Z)\cong H_*(Z;\Z)$ of graded abelian 
  groups. Since $H_*(k;\Q)$ is an isomorphism, there exists a
  maximal set $\cQ$ of primes for which $H_*(k_{\cQ})$ is also an
  isomorphism. Since $H_*(Y;\Z)$ and $H_*(Z;\Z)$ are finitely
  generated in each dimension (and vanish in large dimensions), 
  its complement $\cP$ is finite. If $\cP$ is 
  non-empty, write its elements as $p_1$, \dots, $p_s$ and define 
  $\cQ_i=\cQ\cup\{p_1,\dots ,p_i\}$; otherwise, take $\cQ_0=\cQ$. 
  Since $\cQ_s$ contains \emph{all} primes, it suffices to show 
  that $Y_{\cQ_s}\simeq Z_{\cQ_s}$.
   
  In fact we prove that $Y_{\cQ_i}\simeq Z_{\cQ_i}$ for 
  every~$0\le i\le s$, using induction on $i$. The base case is 
  $i=0$; it holds because $H_*(k)$ is an isomorphism when
  $\cP=\emptyset$, so $k$ is a homotopy equivalence.
  Now assume that $g\colon Y_{\cQ_i}\to Z_{\cQ_i}$ is a homotopy 
  equivalence, and write $p=p_{i+1}$. By choice of~$Y$, there is a 
  homotopy equivalence~$f\colon Y_p\to Z_p$, so we may apply
  the first claim of Lemma~\ref{thm:homotopy-pullback}. This 
  identifies $Y_{\cQ_{i+1}}$ as the homotopy pullback of
\[
\begin{CD}
      Z_{\cQ_i}@>{hl_0}>>Z_0@<l_0<<Z_p
\end{CD}\,,
\]
  where $h$ is the homotopy equivalence
  $f_0g_0^{-1}\colon Z_0\to Z_0$. By assumption, there exist 
  homotopy equivalences $d\colon Z_{p}\to Z _{p}$ and 
  $e\colon Z_{\cQ_i}\to Z_{\cQ_i}$ such that $h\simeq d_0e_0^{-1}$.
  The second claim of Lemma~\ref{thm:homotopy-pullback}
  then confirms that $Y_{\cQ_{i+1}}\simeq Z_{\cQ_{i+1}}$, 
  and completes the inductive step.
\end{proof}

\section{Classification up to homotopy equivalence}\label{sec:homotopy}

Finally, we return to the case of weighted projective space. 

In Theorem~\ref{thm:kawasaki-wps} we selected a generator~$\xi_{1}$ 
for $H^2(\PP(\chi))\cong\Z$. Given any set of primes $\cP$, its localization
in~$H^{2}(\PP(\chi)_{\cP};\Z_{\cP})\cong\Z_{\cP}$ must also be a 
generator. We therefore define the \emph{degree}~$\deg(h)$ of any self-map 
$h$ of  $\PP(\chi)_{\cP}$ to be the $\cP$-local integer satisfying 
$H^*(h;\Z_{\cP})(\xi_1)=\deg(h)\,\xi_1$; this determines a multiplicative 
function 
\begin{equation}
  \label{eq:degree}
  {\deg}\colon [\PP(\chi)_{\cP},\PP(\chi)_{\cP}] \to \Z_{\cP}.
\end{equation}
Remark \ref{rem:simfree} shows that any such $h$ is a homotopy 
equivalence if and only if $H^*(h;\Z_{\cP})$ is an isomorphism. 

 \begin{proposition} \label{thm:degree} \hfill 
  \begin{enumerate}
   \item \label{thm:degree-unit}
    A self-map of~$\PP(\chi)_{\cP}$ is a homotopy equivalence
    if and only if its degree is a unit in~$\Z_{\cP}$.
   \item \label{thm:degree-surjective}
    The degree function~\eqref{eq:degree} is surjective. 
   \item \label{thm:degree-CPn}
    If $\cP$ contains no divisor of any~$\chi_j$, then the degree 
    function is a bijection.
  \end{enumerate}
\end{proposition}
 \begin{proof}
  Since $\deg$ is multiplicative and the degree of the identity map is $1$,
  it maps homotopy equivalences to units. Let 
  $h$ be any self-map of~$\PP(\chi)_{\cP}$, and assume that it has 
  degree~$a$. By Theorem~\ref{thm:kawasaki-wps}, $H^*(h;\Z_{\cP})$ 
  induces multiplication by $a^k$
  on~$H^{2k}(\PP(\chi)_{\cP};\Z_{\cP})\cong\Z_{\cP}$, for 
  every~$1\leq k\leq n$. If $a$ is a unit, then $H^*(h;\Z_{\cP})$ is an 
  isomorphism, so $h$ is a homotopy equivalence. Thus 
  \eqref{thm:degree-unit} holds.
  
  Fix a positive integer~$a$, and define the self-map 
  $m_a\colon\PP(\chi)\to\PP(\chi)$ by raising each homogeneous 
  coordinate to the power~$a$; in particular, write 
  $m'_a\colon\CP^{n}\to\CP^{n}$ for the standard case. Thus 
  $m_a$~and~$m'_a$ commute with the map~$\phi$ of 
  \eqref{eq:definition-phi}, leading to the commutative diagram 
  \begin{equation*}
    \begin{CD}
      H^{*}(\PP(\chi))@>H^*(\phi)>>H^{*}(\CP^{n})\\
      @VH^*(m_a)VV@VVH^*(m'_a)V\\
      H^{*}(\PP(\chi))@>>H^*(\phi)>H^{*}(\CP^{n})
    \end{CD}
  \end{equation*}
  Since $H^2(m'_a)$ is multiplication by~$a$, it follows that 
  $\deg(m_a)=a$. But every element~$c\in \Z_{\cP}$ may be written 
  as a 
  quotient~$c=b/a$ of integers, where $a$ is a positive unit in 
  $\Z_{\cP}$. Then \eqref{thm:degree-surjective} follows 
  from~\eqref{thm:degree-unit}, 
  combined with the observations that complex conjugation on 
  a single coordinate has degree~$-1$, and constant self-maps 
  have degree~$0$.
  
  If $\cP$ contains no divisor of any weight, then
  $\phi_{\cP}\colon\CP^{n}_{\cP}\to\PP(\chi)_{\cP}$ is a homotopy 
  equivalence by Theorem~\ref{thm:kawasaki-wps}. To prove 
  \eqref{thm:degree-CPn}, it therefore suffices to consider maps 
  $h_1$,~$h_2\colon \CP^{n}_{\cP}\to\CP^{n}_{\cP}$ of equal 
  degree; in other words, we may restrict attention to the special case 
  $\CP^n$. Since $\CP^\infty_{\cP}$ is an Eilenberg--Mac\,Lane space~%
  $K(\Z_{\cP},2)$, the compositions of
  $i_{\cP}\colon\CP^n_{\cP}\to\CP^\infty_{\cP}$ with $h_1$ and $h_2$
  are homotopic. Moreover, $\CP^{n}_{\cP}$ is $2n$-dimensional and its
  image is the $(2n+1)$-skeleton of $\CP^\infty_{\cP}$, so the
  homotopy corestricts to a homotopy $h_1\simeq h_2$. Thus
  $\deg$ is injective, and \eqref{thm:degree-CPn} follows.
\end{proof}
The special case $\CP^n$ of part~\eqref{thm:degree-CPn} is well-known
\cite[Thm.~2.2]{McGibbon:1982}, but is stated there without proof.

To complete the proof of Theorem~\ref{thm:genus-wps}, it remains 
only to show that the criteria of Proposition \ref{thm:genus-rigid} 
apply to $\PP(\chi)$.

\begin{proof}[Proof of Theorem~\ref{thm:genus-wps}]
  Let $Y$ be an element of the Mislin genus of~$\PP(\chi)$.
  Since $H_{*}(Y)\cong H_{*}(\PP(\chi))$ as graded abelian groups, 
  $Y$ is homotopy equivalent to a CW complex of dimension~$2n$, 
  by \cite[Prop.~4C.1]{Hatcher:2001}. 
  Furthermore, $H^*(\PP(\chi);\Q)$ is multiplicatively generated by 
  a single element of degree~$2$, so any of the homotopy 
  equivalences $Y_p\simeq\PP(\chi)_p$ induces the corresponding 
  structure on $H^*(Y;\Q)$. A multiplicative generator $\gamma$ may 
  be chosen to be integral in $H^2(Y;\Q)$, because the Universal 
  Coefficient Theorem confirms that $H^2(Y;\Z)\cong\Z$. Also, 
  $\gamma$ is represented by a 
  map~$j\colon Y\to\CP^\infty\simeq K(\Z,2)$, for which $H^2(j;\Z)$ 
  is an isomorphism. Up to homotopy, $j$ factors through 
  $\CP^{n}\subset\CP^{\infty}$, so its corestriction 
  $j'\colon Y\to\CP^n$ is a rational homotopy equivalence.
  Since $\phi\colon\CP^n\to\PP(\chi)$ is a rational homotopy 
  equivalence by Theorem~\ref{thm:kawasaki-wps}, the same holds 
  for the composition $\phi j'\colon Y\to\PP(\chi)$. Criterion (i) 
  of Proposition \ref{thm:genus-rigid} is therefore satisfied by 
  $k=\phi j'$.
  
  Now let $h\colon \PP(\chi)_0 \to \PP(\chi)_0$ be a homotopy 
  equivalence, let $\deg(h)=\pm a/b$ where $a,b\in \N$, and 
  let $\cP$ and $\cQ$ 
  be two disjoint sets of primes.
  Write $a=a'a''$ and $b=b'b''$, where $a'$, $b'$ are divisible 
  only by primes not contained in $\cP$, and $a''$, $b''$ are 
  divisible only by primes contained in $\cP$. 
  Then $a'/b'\in \Z_{\cP}$ and $b''/a''\in \Z_{\cQ}$ are units. 
  So Proposition~\ref{thm:degree}\,\eqref{thm:degree-surjective}
  guarantees the existence of homotopy equivalences 
  $d\colon \PP(\chi)_{\cP}\to \PP(\chi)_{\cP}$ and 
  $e\colon \PP(\chi)_{\cQ}\to \PP(\chi)_{\cQ}$
  of degrees $\pm a'/b'$~and~$b''/a''$ respectively, and
  $h\simeq d_0e_0^{-1}$ by 
  Proposition~\ref{thm:degree}\,\eqref{thm:degree-CPn}.
  Criterion (ii) of Proposition \ref{thm:genus-rigid} is
  therefore satisfied, as required.
\end{proof}  
 
\begin{proof}[Proof of Theorem~\ref{thm:classification-homotopy}]
  If $\chi$~and~$\chi'$ have the same $p$-content up to order, then
  some permutation of homogeneous coordinates defines a
  homeomorphism~$\PP(\pcont{\chi}{p})\cong\PP(\pcont{\chi'}{p})$
  for each prime~$p$. This homeomorphism may be localized at~$p$.

  Now consider the map 
\begin{equation*}
   g\colon\PP(\pcont{\chi}{p})\to\PP(\chi), \quad
   [z_{0}:\dots:z_{n}]\mapsto
   [z_{0}^{\alpha(0)}:\dots:z_{n}^{\alpha(n)}],
\end{equation*}
  where $\alpha(j)=\chi_j/\pcont{\chi}{p}_j$ for~$0\leq j\leq n$.
  Theorem~\ref{thm:kawasaki-wps} implies that $H^*(g;\Z_p)$ is an 
  isomorphism, and Remarks \ref{rem:simfree} confirm that $g_p$ is 
  a homotopy equivalence. So $g_p^{-1}$~and~$g'_p$ determine a 
  chain of maps
  \begin{equation*}
    \PP(\chi)_p\simeq\PP(\pcont{\chi}{p})_p\cong
    \PP(\pcont{\chi'}{p})_p\simeq\PP(\chi')_p
  \end{equation*}
  for any prime~$p$, and the result follows from Theorem~\ref{thm:genus-wps}.
\end{proof}

\subsection*{Acknowledgements}
The authors are particularly grateful to the referee, whose suggestions led
to several improvements in the structure and exposition of this work. A.\,B.\ was
supported  in part by a Rider University Summer Research Fellowship and
Grant \#210386 from the Simons Foundation, and M.\,F.\ by an NSERC Discovery 
Grant.

\end{document}